\documentclass[12pt,a4paper,reqno]{amsart}
\usepackage{amsmath,amssymb,amsthm,amscd,mathrsfs} 
\usepackage[T1]{fontenc}
\usepackage[latin1]{inputenc}
\usepackage[english]{babel}
\usepackage[mathscr]{eucal}
\usepackage{graphicx}
\usepackage{float}
\usepackage{tikz}
\usetikzlibrary{positioning}
\usetikzlibrary{shapes.geometric}
\usetikzlibrary{fit}
\usetikzlibrary{patterns}
\usepackage{caption}
\usepackage{tikz-cd}
\usepackage{cases}
\usepackage{hyperref}
\usepackage{chngcntr}
\usepackage{enumitem}
\usetikzlibrary{decorations.text}

\input{comment.sty}
\includecomment{FZ}
\includecomment{MM}
\counterwithin{equation}{section}
\renewcommand{\th}{\theta}

\renewcommand{\l}{\lambda}

\newcommand{\n}{\nu}

\newcommand{\ph}{\phi}
\def\ph{\phi}
\def\md#1{\ \mbox{\rm(mod }{#1})}

\def\npp#1{N_{\ph}^+(#1)}
\def\ol{\overline}

\def\ph{\phi}
\newcommand{\Q}{{\mathbb Q}}
\newcommand{\Z}{{\mathbb Z}}

\newcommand{\F}{{\mathbb F}}

\def\md#1{\ \mbox{\rm(mod }{#1})}

\def\npp#1{N_{\ph}^+(#1)}

\newcommand{\aF}{\mathfrak a}

\newcommand{\pF}{\mathfrak p}

\newtheorem{theorem}{Theorem}[section]

\newtheorem{lemma}[theorem]{Lemma}
\newtheorem{corollary}[theorem]{Corollary}

\theoremstyle{definition}

\newtheorem{definitions}[theorem]{Definitions}

\theoremstyle{remark}
\newtheorem{example}[theorem]{Example}

\newtheorem{remarks}[theorem]{Remarks}

\begin{document}
\title[]{The index of the   septic  number field defined by $x^7+ax^5+b$}
\textcolor[rgb]{1.00,0.00,0.00}{}
\author{  Hamid Ben Yakkou  }\textcolor[rgb]{1.00,0.00,0.00}{}
\address{Faculty of Sciences Dhar El Mahraz, P.O. Box  1874 Atlas-Fes , Sidi mohamed ben Abdellah University,  Morocco}\email{beyakouhamid@gmail.com }
%\address{Department of Mathematics, Faculty of Sciences and Technology, Moulay Ismail University, Er-rachidia, Morocco}\email{ abderazak.soullami@usmba.ac.ma}
\keywords{Monogenity, Power integral basis, Theorem of Ore, prime ideal factorization, common index divisor} \subjclass[2010]{11R04,
11R16, 11R21}
%\date{\today}
\maketitle
\vspace{0.3cm}
\begin{abstract}  
 Let $K $ be a septic number field generated by   a complex   root $\th$ of a monic irreducible trinomial  $ F(x)= x^7+ax^5+b \in \Z[x]$.  Let $i(K)$ be the index of $K$.  In this paper,  we show that  $i(K) \in \{1, 2, 4\}$. In a such way, we  answer to Problem $22$ of Narkiewicz \cite{Na}  for these number fields. In particular, we provide sufficient conditions for which $K$ is non-monogenic. We illustrate our results by some computational examples.
 %  At the end,	some numerical examples illustrating our theoretical results are given.
\end{abstract}
\maketitle
%\tableofcontents
%\newpage
%\section{Introduction}
\section{Introduction and statements of earlier  results}
Let $K $ be a number field generated by $\th$, a root of a monic irreducible polynomial $F(x)\in \Z[x]$ of degree $n$, and $A_K$  its ring of integers.  By \cite[Theorem 2.10]{Na}, $A_K$ is a free $\Z$-module of rank $n$.   For every $\eta \in A_K$ generating $K$, let  $(A_K: \Z[\eta])= |A_K/\Z[\eta]|$  be the index of $\Z[\eta] $ in $A_K$,  called the index of $\eta$. The ring $A_K$ is said to have a power integral basis if it admits  a $\Z$-basis $(1,\eta,\ldots,\eta^{n-1})$ for some $\eta\in A_K$; $A_K=\Z[\eta]$.  In a such case,  $K$ is said to be monogenic. Otherwise, $K$ is said not monogenic.

% The minimal index of the filed $K$ is defined by:
%	\begin{eqnarray*}\label{m(K)}
%		m(K) = \min \ \{ ( A_K : \Z [\eta]) \, |\,\eta \in A_K \,\, \mbox{and}\,\, K= \Q(\eta)  \}.
%	\end{eqnarray*}
%	So, the field $K$ is monogenic if and only $m(K)=1$. 
In this paper, $i(K)$ will denote  the  index of   $K$, defined as follows: 
\begin{equation}\label{Defi(K)}
i(K) = \gcd \ \{(A_K:\Z[\eta]) \,| \ \eta \in A_K  \, \mbox{and}\,  K=\Q(\eta)\}. 
\end{equation} 
A rational prime $p$ dividing $i(K)$ is called a common index divisor of $K$ (or nonessential discriminant divisor of $K$). Note that if $K$ is monogenic, then  $i(K) = 1$. But, if $i(K)>1$; equivalently if there exist a rational  prime $p$ dividing $i(K)$,  then  $K$ is not monogenic.  The canonical examples of monogenic number fields are quadratic and cyclotomic fields. 
%Indeed, any quadratic field $K=\Q(\sqrt{d})$, with $d\in \Z$ and square free, has ring of integers given by: $$A_K=
%\left
%\{\begin{array}{ll} \Z[\sqrt{d}],& \mbox{if } d \equiv 2, 3 \md 4,\\
%\Z[\frac{1+\sqrt{d}}{2}],&\mbox{if } d \equiv 1 \md 4.
%\end{array}
%\right.$$ Also, if $K=\Q(\xi_n)$ is any cyclotomic number field, where $n \in \N$ and $\xi_n$ is a primitive  nth root of unity, then $A_K = \Z[\xi_n]$.  
 The first non-monogenic number field was given by Dedekind in 1878. He showed that the cubic number field $\Q(\th)$ is not monogenic when $\th$ is a root of the polynomial $x^3-x^2-2x-8$ (cf. \cite[p.64]{Na}). More precisely, he showed that $2$ is a common index divisor of $K$.  The problem of testing the monogenity of number fields and constructing power integral bases have been intensively studied. There are extensive   results regarding this problem. These results were  treated by using  different approachs. Ga\'{a}l, Gy\H{o}ry,  Peth\H{o}, Pohst,  Remete (cf. \cite{Gacta2, G19, Ga21,  Gacta1, GR17,Gyoryredecide, PP}) and their  research teams who succeeded to study the monogenity of several number fields, are  based in their approach on the  resolution of  index form equations. In  \cite{Gyoryredecide},  Gy\H{o}ry   gave effective bounds of the  integral solutions of index form equations, see also \cite{Gyoryr1983}. In \cite{GP}, Ga\'{a}l,  Peth\H{o} and Pohst studied indices  of quartic number fields.  In \cite{DS}, Davis and Spearman showed that  the index of a  quartic number  field defined by $x^4+ax+b$ contained in the set $\{1, 2, 3, 6\}$.   In\cite{Gacta1}, Ga\'{a}l and Gy\H{o}ry described an algorithm to solve  index form equations in quintic number fields  and they computed all generators of power integral bases in  totally real quintic number fields with Galois group $S_5$. In  \cite{Gacta2}, Bilu, Ga\'{a}l and Gy\H{o}ry   studied  the monogenity of    totally real sextic number fields with Galois group $S_6$. In \cite{GR17},  Ga\'al and Remete studied the monogenity of pure number  fields  $\Q(\sqrt[n]{m}),$ where  $3\leq n\leq 9$ and $m$ is square free. They  also showed in \cite{GRN} that for a square free rational integer  $m \equiv 2, 3  \md{4}$,  the octic number filed $\Q(i, \sqrt[4]{m})$ is not monogenic.   In \cite{ PP},   Peth\H{o} and  Pohst studied indices in multiquadratic number fields.  Also, in \cite{PethoZigler}, Peth\H{o} and Ziegler gave an efficient criterion to decide whether the maximal order of a biquadratic field has a unit power integral basis or not.
% For important theoretical results about the  monogenity of relative extensions, see  \cite{Gyoryrelative} by Gy\H{o}ry.    
The books \cite{EG} by Evertse and   Gy\H{o}ry, and \cite{G19} by  Ga\'al  give detailed surveys on the discriminant, the index form theory and its applications, including related Diophantine equations and monogenity of number fields. Nakahara's research team   based on the existence of relative  power  integral bases of some special  sub-fields, they studied the monogenity of several number fields:   Ahmad,  Nakahara and  Husnine \cite{ANHN} proved that for a square free rational integer $m$, if $m \equiv 2, 3 \md 4$ and $m \not \equiv  \pm 1 \md 9, $ then the sextic pure field $K = \Q(\sqrt[6]{m})$ is monogenic. But, it is not monogenic when $m\equiv 1\md{4}$, see \cite{AN} by Ahmad,  Nakahara and  Hameed. In \cite{Smtacta},  Smith studied the monogenity of radical extensions and  he  gave  sufficient conditions for a  Kummer extension to be not monogenic.
%  Recall that the polynomial $F(x)$ is said to be monogenic if $A_K = \Z[\th]$; that is  $(1, \th, \ldots, \th^{n-1})$ form  a power integral  basis  of $A_K$.  Note that the monogenity of  the polynomial $F(x)$ implies the monogenity of the field $K$. But, the converse is not true, because a number field defined by non-monogenic polynomial can be monogenic;  the non-monogenity of $F(x)$ does not imply the non-monogenity of $K$ (see Theorem 2.1  of  \cite{BFT1} by Ben Yakkou and El Fadil,  which gives infinite parametric monogenic number field defined by non-monogenic trinomials). Let us recall some result concerning the monogenity of trinomials. 
%In \cite{JK}, Jhorar and Khanduja gave necessary and sufficient conditions involving only on $a$ and $b$ for which the trinomial  $x^n+ax+b$ is monogenic.
In \cite{JW}, Jones and White gave infinite parametric families   of monogenic trinomials
%I replace trinomilas by trinomials
with non square free discriminant. 

%Also, in \cite{JP}, Jones and Phillips construct infinite families of  monogenic trinomials defined by $F_n(x)=x^n+a(m,n)x+b(m,n)$ where $a(m,n)x$ and $b(m,n)$ are certain prescribed forms in $m$. 
Recently (2021),   Ga\'al studied the monogenity of sextic number fields defined by $x^6+ax^3+b$ (see \cite{Ga21}). Based on the important works of \O. Ore, Gu\`{a}rdia,  Montes and  Nart about the application of   Newton polygon techniques  on the factorization of ideals of the ring $A_K$ into a product of powers of prime ideals (see \cite{Nar, Narprime, MN92, O}), several results obtained about the monogenity and indices of number fields defined by trinomials. 
%For $x^5+ax+b$, see \cite{JSK} by Jakhar, Kaur and Kumar. 
For $x^7+ax^3+b$, see  \cite{FKcom} by  El Fadil and Kchit. For  $x^8+ax+b$, see  \cite{BATCA} by  Ben Yakkou. Also, in \cite{BRM}, we   studied the non-monogenity of number fields defined by $x^n+ax^m+b$. 

The goal of the present paper is to calculate the index  of the  number field generated by a root  of a monic irreducible trinomial of type  $F(x)=x^7+ax^5+b$.  Note that, the above mentioned   results  did not answer the question  of  the monogenity  and indices of number fields defined  by $x^{7}+ax^5+b$. Note also  that  in the case of the septic number fields defined by $x^7+ax^5+b$, for the moment we don't have any general practical procedure for  solving the corresponding index form equations $I(x_2, \ldots, x_7)=\pm 1$. For this reason, we have based our method on prime ideal factorization  techniques. 
\section{Main results}
In what follows, let  $K $ be a number field generated by $\th$, a root of a monic irreducible trinomial $F(x)= x^7+ax^5+b \in \Z[x]$, where 
$ab \neq 0$ and $A_K$  its ring of integers. For every rational prime 
$p$ and any non-zero  $p$-adic integer $m$, $\n_p(m)$  denote the $p$-adic valuation of $m$; the highest power of $p$ dividing $m$, and $m_p := \frac{m}{p^{\n_p(m)}}$. Without loss of generality, for every prime $p$, we assume that \begin{eqnarray}\label{Hypothese}
\n_p(a)<2 \,\, \text{or} \,\, \n_p(b)<7.
\end{eqnarray}
%To explain this,  suppose that $\n_p(a) \ge 2$ and $\n_p(b)\ge 7$. Let $\n_p(a)=2q_1+r_1$ and $\n_p(b)=7q_2+r_2,$ where $r_1 \in \{0, 1\}$ and $1 \le r_2 \le 6$. Let $q=\min\{q_1, q_2\}, \eta=\frac{\th}{p^q}, A=\frac{a}{p^{2q}}, B=\frac{b}{p^{7q}},$ and $G(x)=x^7+Ax^5+B$. Then, the following hold:
%\begin{enumerate}
%	\item  $\n_p(A)<2$  or $\n_p(B)<7$
%	\item $G(x)$ irreducible over $\Q$ and $G(\eta)=0$.
%	\item $K=\Q(\th)=\Q(\eta)$.
%\end{enumerate}So, up to replace $F(x)$ by $G(x)$, the claim holds.\\
For the simplicity, if $pA_K=\pF_1^{e_1}\cdots\pF_g^{e_g}$ is the   factorization of $pA_K$ into a product of powers of prime ideals in $A_K$ with residue degrees $f(\pF_i/p)=[A_K/\pF_i : \Z/p\Z]=f_i$, then we write $pA_K = [f_1^{e_1}, \ldots, f_g^{e_g}]$. Also, if $e_i=1$ for some $i$, then we shortly write $f_i$ instead of $f_i^{e_i}$. Also, by  the Fundamental Equality (\cite[Theorem 4.8.5]{Co}), one has: 
\begin{eqnarray}\label{FE}
\sum_{i=1}^{g}e_if_i=7=\deg(K).
\end{eqnarray}
  In this paper we prove the following results. 
  \begin{theorem}\label{pge3}    Let $K= \Q(\th)$ be a number field with $\th$ a root of a monic  irreducible polynomial $F(x)=x^7+ax^5+b \in \Z[x]$.  Then for any odd rational  $p$, $p$ is not a common index divisor of $K$; $p$ does not divide $i(K)$.
  \end{theorem} 
From the above theorem, the only candidate rational prime to divide $i(K)$ is $2$.   Thus, either $i(K)=1$ or $i(K)=2^k$ for some positive integer $k$. The following result gives the complete answer.
\begin{theorem}\label{p=2}  Let $K$ be a number field generated by a complex root $\th$ of a monic irreducible trinomial $F(x)=x^7+ax^5+b \in \Z[x]$.  	Then  Table \ref{table1} gives  the form of  factorization of the ideal  $2A_K$ into a product of powers of prime ideals of $A_K$, and the value of the index  $i(K)$ in every case. Furthermore, $2$ is common index divisor of $K$ if and only if one of the conditions $C9, C10, C11, C17$ holds.
\end{theorem}
\begin{table}[h!]
	\centering
	\begin{tabular} { | c | c | c | c|  }   
		\hline
		Case & Conditions & Factorization of $2A_K$& $i(K)$  \\
		\hline
		C1 & $a \equiv 1 \md{2}$ and $b \equiv 1 \md{2}$ & $[2^1,5^1]$& $1$  \\
		\hline
			C2 & $a \equiv 0 \md{2}$ and $b \equiv 1 \md{2}$ & $[1, 3, 3]$& $1$  \\
		\hline
			C3 & $7\n_2(a)>2\n_{2}(b)$ and $\n_2(b) \in \{1, 2, 3, 4, 5, 6\}$ &$[1^7]$& $1$ \\
		\hline
		C4 & $\n_2(a)=1, \n_2(b)\ge 4 $ and $5 \nmid \n_2(b)$ &$[1^2, 1^5]$& $1$ \\
		\hline
		C5 & $\n_2(a)=1$ and $5 \mid \n_2(b)$ &$[1, 1^2, 4]$& $1$ \\
		\hline
			C6 & $a\equiv 3 \md{8}, b\equiv 0 \md{8}$, and $5 \nmid \n_2(b)$ &$[1^5, 2]$& $1$ \\
		\hline
			C7 & $a\equiv 3 \md{8}, b\equiv 0 \md{8}$, and $5 \mid \n_2(b)$ &$[1, 2, 4]$& $1$ \\
		\hline
			C8 & $a\equiv 7 \md{8}, b\equiv 4 \md{8}$ &$[2, 1^5]$& $1$ \\
		\hline
			C9 & $a\equiv 3 \md{8}, b\equiv 4 \md{8}$ &$[1, 1, 1^5]$& $2$  \\
		\hline
			C10 & $a\equiv 7 \md{8}, b\equiv 0 \md{8}$, and $5 \nmid \n_2(b)$ &$[1, 1, 1^5]$ & $2$ \\
		\hline
			C11 & $a\equiv 7 \md{8}, b\equiv 0 \md{8}$, and $5 \mid \n_2(b)$ &$[1, 1, 1, 4]$ &  $2$\\
		\hline
			C12 & $a \equiv 3 \md{4}$ and $b \equiv 2 \md{4}$ & $[1^2, 1^5]$ & $1$  \\
		\hline
			C13 & $a \equiv 1 \md{4}, b \equiv 0 \md{4}$ and $5 \nmid \n_{2}(b)$ & $[1^2, 1^5]$ & $1$  \\
		\hline
		C14 & $a \equiv 1 \md{4}, b \equiv 0 \md{4}$ and $5 \mid \n_{2}(b)$ & $[1, 1^2,  4]$ & $1$  \\
		\hline
			C15 &  $(a, b) \in \{(1, 10), (9, 2), (1, 6), (9, 14)\} \md{16}$ & $[1^2,  1^5]$ & $1$  \\
		\hline
			C16 & $(a, b) \in \{(1, 18), (17, 2), (1, 14), (17, 30)\} \md{32}$  & $[2,  1^5]$ &  $1$ \\
		\hline
		C17 & $(a, b) \in \{(1, 2), (17, 18), (1, 30), (17, 14)\}\md{32}$ & $[1, 1,  1^5]$ &  $2$ or $4$ \\
	%	\hline
		%	C18 & $a \equiv 17 \md{32}$ and $b \equiv 14, 18 \md{32}$ & $[1, 1,  1^5]$ &  $\in \{2 , 4\}$ \\
		%\hline
		%	C19 & $a \equiv 17 \md{32}$ and $b \equiv 2, 30 \md{32}$ & $[2,  1^5]$ &  $1$ \\
		%\hline
		%	C20 & $a \equiv 9 \md{16}$ and $b \equiv 2, 14 \md{16}$ & $[1^2,  1^5]$ &  $1$ \\
		\hline
			C18 & $(a, b) \in \{(5, 2), (5, 14), (13, 6), (13, 10)\}\md{16}$ & $[1^5,  1^5]$ &  $1$ \\
			\hline
		
	\end{tabular}
	\caption{The factorization of $2A_K$ and the value of $i(K)$.}
	\label{table1}
\end{table}
\begin{corollary}\
	\\Let $K=\Q(\th)$ be a number field with $\th$ a root of a monic irreducible polynomial $F(x)=x^7+ax^5+b \in \Z[x]$.   Then
	\begin{enumerate}
		\item   $i(K)>1$ if and only if one of the conditions  $C9, C10, C11, C17$ hold. Otherwise, $i(K)=1$. 
		\item If any of the  conditions   $C9, C10, C11, C17$ holds, then $K$ is not monogenic; $\Z_K$ has no power integral basis. 
	\end{enumerate}
\end{corollary}
\begin{remarks}\
	The condition $i(K)=1$ is not sufficient for $K$ to be monogenic. Indeed there exist non-monogenic number fields their index equal $1$. For example, consider the pure cubic number field  $K=\Q(\sqrt[3]{m}),$ where $m=1+9k$ and $k=10$ or $12$. According to the results of   \cite{GR17},  by Ga\'al and Remete, $i(K)=1$, and the corresponding  index form equation  is $I(x_2, x_3)= 3x_1^2+3x_1^2x_2+x_1x_2^2-kx_2^3 \pm 1$  has no integral solutions. Hence, $K$ is not monogenic.
	% For an other example, see \cite[Example 7.4.4]{Alacawiliamsbook}.
\end{remarks}
%\begin{theorem}\label{npib3} Let $K=\Q(\th)$ and $F(x)$ be as in Theorem  \ref{th1}. If $a \equiv 8 \md 9$ and $b\equiv 0 \md 9$, then $2$  is a common index divisor of $K$. In particular, the field $K$ is not monogenic.
%\end{theorem}

%\begin{theorem}\label{3does}
%	Let $K=\Q(\th)$ and $F(x)$ be as in Theorem  \ref{th2}.  If $a\equiv 1\md 3$ and $b\equiv 0 \md{3}$, then $3$ does not divide $i(K)$.
%\end{theorem}

%\section{Preliminary results} 
\section{Examples}
To illustrate our results, we propose some examples. 
Let $K=\Q(\th)$ be a septic number field with $\th$ a root of a monic irreducible polynomial $F(x)=x^7+ax^5+b \in \Z[x]$. 
\begin{enumerate}
\item Let $F(x)=x^7+867x^5+68$. Since $F(x)$ is a $17$-Eisenstein polynomial,  it is irreducible over $\Q$. By Case $C9$ of Table \ref{table1} of Theorem \ref{p=2}, $i(K)=2$. 	So, $K$ is not monogenic.
\item  Let $F(x)=x^7+45927x^5+24$. The polynomial $F(x)$ is irreducible over $\Q$ as it is a $3$-Eisenstein polynomial. In view of Case $C10$ of Table \ref{table1} of Theorem \ref{p=2}, $i(K)=2$. Consequently, $A_K$ has no power integral basis.
\item  Let $F(x)=x^7+33x^5+66$. As $F(x)$ is a $11$-Eisenstein polynomial, it irreducible over $\Q$. According to Case $C17$ of Table  \ref{table1}, $K$ is not monogenic and $i(K) \in \{2, 4\}$.
\item  Let $F(x)=x^7+p^rx^5+p$, where $p$ is an odd rational prime and $r$ is a positive integer. By Theorem \ref{pge3} and Case $C1$ of Table \ref{table1}, $i(K)=1$.
\end{enumerate}
\section{Preliminary results}
Let $K $ be a number field generated by $\th$, a root of a monic irreducible trinomial $F(x)= x^7+ax^5+b \in \Z[x]$ and $A_K$  its ring of integers. Let $p$ be a rational prime.  We start by recalling  the  following Lemma  which gives a necessary and sufficient condition for a rational prime  $p$ to be a prime common index divisor of $K$. This Lemma will  play an important role in the proof of our results (see \cite[Theorems 4.33 and 4.34 ]{Na} and \cite{R}). 
\begin{lemma} \label{comindex}
	Let  $p$ be a rational prime  and $K$  a number field. For every positive integer $f$, let $L_p(f)$ denote the number of distinct prime ideals of $A_K$ lying above $p$ with residue degree $f$ and $N_p(f)$ denote the number of monic irreducible polynomials of  $\F_p[x]$ of degree $f$. Then $p$ is a common index divisor of $K$ if and only if $L_p(f) > N_p(f)$ for some positive integer $f$.
\end{lemma}
%\begin{remark}\label{remark}\
%	\\	\begin{enumerate}
%Recall  by \cite[Proposition 4.35]{Na} that  the number of  monic irreducible polynomials of degree $f$ in $\F_p[x]$ is
%$$N_p(f) = \frac{1}{f} \sum_{d \mid f} \mu (d) p^{\frac{f}{d}},$$
%where $\mu$ is the M\"{o}bius function. Note  also that $N_{p}(f+1)>N_{p}(f)$ (except one case when $p=2$ and $f=1$).  
%	 	Note that the condition $i(K)=1$ is not sufficient for the monogenity  of $K$. The pure cubic number field  $K=\Q(\sqrt[3]{175})$ is a simple example of the case $i(K)=1$, but $K$ is not monogenic as its index form equation equals $5x^3-7y^3$ and never assumes the values $\pm 1$. 
To apply the above  Lemma \ref{comindex}, we will need to determine the number of distinct prime ideals of $A_K$ lying above $p$. We will use Newton polygon techniques. So, let us  shortly recall   some fundamental notions and results on this method. For more details, we refer to \cite{ Nar, Narprime, MN92, O}.  The reader can see also \cite{BRM, BATCA}.  
Let $p$  be a rational prime  and $\n_p$  the discrete valuation of $\Q_p(x)$  defined on $\Z_p[x]$ by $$\n_p\left(\sum_{i=0}^{m} a_i x^i\right) = \min \{ \n_p(a_i), \, 0 \le i \le m\}.$$ Let $\phi(x) \in \mathbb{Z}[x]$ be a monic polynomial whose reduction modulo $p$ is irreducible. The polynomial $F(x) \in \mathbb{Z}[x]$ admits a unique $\phi$-adic development $$ F(x )= a_0(x)+ a_1(x) \phi(x) + \cdots + a_n(x) {\phi(x)}^n,$$ with $ \deg \ ( a_i (x))  < \deg \ ( \phi(x))$. For every $0\le i \le n,$   let $ u_i = \n_p(a_i(x))$. The $\phi$-Newton polygon of $F(x)$ with respect to $p$ is the  lower boundary convex envelope of the set  of points $ \{  ( i , u_i) \, , 0 \le i \le n \, , a_i(x) \neq 0  \}$ in the euclidean plane, which we denote   by $N_{\phi} (F)$. The polygon  $N_{\phi} (F)$ is the union of different adjacent sides $ S_1, S_2, \ldots , S_g$ with increasing slopes $ \lambda_1, \lambda_2, \ldots,\lambda_g$. We shall write $N_\phi(F) = S_1+S_2+\cdots+S_g$. The polygon determined by the sides of negative slopes of $N_{\phi}(F)$ is called the  $\phi$-principal Newton polygon of $F(x)$  with respect to $p$ and will be denoted by $\npp{F}$. The length of $\npp{F}$ is $ l(\npp{F}) = \nu_{\overline{\ph}}(\overline{F(x)})$;  the highest power of $\phi$ dividing $F(x)$ modulo $p$.\\
Let $\mathbb{F}_{\phi}$ be the finite field   $ \mathbb{Z}[x]\textfractionsolidus(p,\phi (x)) \simeq \mathbb{F}_p[x]\textfractionsolidus (\overline{\ph(x)}) $.
We attach to  any abscissa $ 0 \leq i \leq  l(\npp{F})$ the following residue coefficient $ c_i \in  \mathbb{F}_{\phi}$
:

$$c_{i}=
\left
\{\begin{array}{ll} 0,& \mbox{if }  (i , u_i ) \, \text{ lies  strictly  above }  \ \npp{F},\\
\dfrac{a_i(x)}{p^{u_i}}
\,\,
\md{(p,\phi(x))},&\mbox{if } \  (i , u_i ) \, \text{lies on } \npp{F}. 
\end{array}
\right.$$Let $S$ be one of the sides of $\npp{F}$. Then the length of $S$, denoted $l(S)$ is the length of its  projection to the horizontal axis and its height, denoted $h(S)$ is the length of its  projection to the vertical axis.  Let   $\lambda = - \frac{h(S)}{l(S)}= - \frac{h}{e} $  its slope, where $e$ and $h$ are two positive coprime integers.  The degree of $S$ is $d(S) = \gcd(h(S), l(S))=  \frac{l(S)}{e}$; it is equal to the the number of segments into which the integral lattice divides $S$. More precisely, if $ (s , u_s)$ is the initial point of $S$, then the points with integer coordinates  lying in $S$ are exactly $$  (s , u_s) ,\ (s+e , u_s - h) , \ldots, (s+de , u_s - dh).$$  The natural integer $e= \frac{l(S)}{d(S)}$ is called the ramification index of the side $S$ and denoted by $e(S)$. We attach to $S$ the following residual polynomial:  $$ R_{\l}(F)(y) = c_s + c_{s+e}y+ \cdots + c_{s+(d-1)e}y^{d- 1}+ c_{s+de}y^d \in \mathbb{F}_{\phi}[y].$$  
Now, we give some related definitions.
\begin{definitions}
	Let $F(x) \in \Z[x]$ be a monic irreducible polynomial. Let $F(x) \equiv \prod_{i=1}^t \ph_i(x)^{l_i} \md{p}$ be  the factorization  of $F(x)$ into a product  of powers of distinct monic irreducible polynomials in  $\mathbb{F}_p [x]$. For every $i=1,\dots,t$, let  $N_{\ph_i}^+(F)=S_{i1}+\dots+S_{ir_i}$, and for every {$j=1,\dots, r_i$},  let $R_{\l_{ij}}(F)(y)=\prod_{s=1}^{s_{ij}}\psi_{ijs}^{n_{ijs}}(y)$ be the factorization of $R_{\l_{ij}}(F)(y)$ in $\F_{\ph_i}[y]$. 
	\begin{enumerate}
		\item  For every $i=1,\dots,t$, the $\ph_i$-index of $F(x)$, denoted by $ind_{\ph_i}(F)$, is  deg$(\ph_i)$ multiplied by  the number of points with natural integer coordinates that lie below or on the polygon $N_{\ph_i}^{+}(F)$, strictly above the horizontal axis  and strictly beyond the vertical axis.
		\item The polynomial $F(x)$ is said to be $\phi_i$-regular with respect to $\n_p$ (or $p$) if for every $j=1,\dots, r_i$,  $R_{\l_{ij}}(F)(y)$ is separable; $n_{ijs}=1$.
		\item The polynomial $F(x)$ is said to be $p$-regular if it is $\phi_i$-regular for every $ 1 \leq i \leq t$.
	\end{enumerate}
\end{definitions} 
Now, we recall Ore's theorem which will be  used in the proof of  Theorems \ref{pge3}  and \ref{p=2} (see   \cite[Theorems 1.13, 1.15 and  1.19]{Nar},  \cite{MN92} and \cite{O}).
\begin{theorem}\label{ore} (Ore's Theorem)\
	\\Let $K$ be a number field generated by  $\th$, a root of a monic irreducible polynomial $F(x) \in \Z[x]$.  Under the above notations, we have: 
	\begin{enumerate}
		\item
		$$ \nu_p((A_K:\Z[\th]))\ge \sum_{i=1}^t ind_{\ph_i}(F).$$ Moreover, the  equality holds if $F(x)$ is $p$-regular
		\item
		If  $F(x)$ is $p$-regular, then 
		$$pA_K=\prod_{i=1}^t\prod_{j=1}^{r_i}
		\prod_{s=1}^{s_{ij}}\pF^{e_{ij}}_{ijs},$$ where $e_{ij}$ is the ramification index
		of the side $S_{ij}$ and $f_{ijs}=\mbox{deg}(\ph_i)\times \mbox{deg}(\psi_{ijs})$ is the residue degree of $\mathfrak{p}_{ijs}$ over $p$.
	\end{enumerate}
\end{theorem} 
The following result is an immediate consequence of the above theorem. 
\begin{corollary}\label{corollaryore} Under the above hypotheses, then the following hold:
	\begin{enumerate}
		\item  If for some $i=1, \ldots, r$, $l_i=1$, then the factor $\ph_i(x)$ of $F(x)$ modulo $p$ provides a unique prime ideal of $A_K$ lying above $p$, of residue degree equals $\deg(\ph_i(x))$ and of  ramification index  equals $1$.
		\item If for some $i=1, \ldots, r$, $N_{\ph_i}^+(F)=S_{i1}+S_{i2}+\cdots+S_{ik}$ has $k$  sides of degree $1$ each, then the factor $\ph_i(x)$ of $F(x)$ modulo $p$ provides $k$ prime ideals of $A_K$ lying above $p$ with the same  residue degree equals $\deg(\ph_i(x))$, and  of ramification indices $e(\pF_{ij1}/p)=e(S_{ij}), \, j=1, \ldots, k$.
	\end{enumerate}
	
\end{corollary}
The study of the following example is based on Newton polygon techniques.
\begin{example}
	Consider the monic irreducible polynomial $F(x) = x^9 +54 x +134 $, and let $K=\Q(\th)$ with $\th$ a root of $F(x)$.  For $p=3$, we have  $	F(x) \equiv \ph_1(x)^9 \md{3}$, where $\phi= x+2$. The $\ph$-adic development of $F(x)$  is
	\begin{eqnarray*}
	F(x)&=&-486+2358\ph(x)-4608\ph(x)^2+5376\ph(x)^3-4032\ph(x)^4+2016\ph(x)^5-672\ph(x)^6 \nonumber   \\
	&+&144\ph(x)^7- 18\ph(x)^8+\ph(x)^9.
	\end{eqnarray*}
	Thus, $N_{\ph_1}^{+}(F)= S_{11}+S_{12}+S_{13}$ with respect to $\n_3$,  has three sides  of degree $1$ each joining the points $(0,5),(1,2),(3,1)$, and $(9,0)$ in the Euclidean plane  with respective slopes $\l_{11} =-3$, $\l_{12}=\frac{-1}{2}$, and $\l_{13}=\frac{-1}{6}$ (see FIGURE 1).     Their attached residual polynomials
	$R_{\l_{1k}}(F)(y), k=1,2,3$,  are irreducible in $\F_{\ph}[y] \simeq \F_3[y] $ as they are of degree $1$ each. Thus, $F(x)$ is $\ph_i$-regular.  Hence it is $3$-regular. By  Theorem \ref{ore}, we have
	$$\nu_2((A_K:\Z[\th])) = ind_{\ph_1}(F) = \deg(\ph_1) \times 4 = 4 $$ and $$3A_K = \pF_{111} \cdot  \pF_{121}^2  \cdot  \pF_{131}^6,$$ where $f(\pF_{1k1}/3)= 1$ for $k=1,2,3$. Thus, for $p=2$, we have  $L_2(1)=3 > N_2(1)=2$. By Lemma \ref{comindex}, $2$ divides $i(K)$. Consequently, $K$ is not monogenic. 
	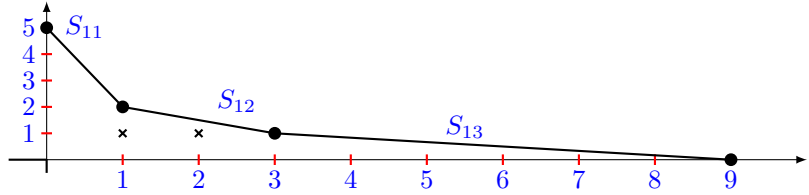
\begin{figure}[htbp] 
		\centering
		\begin{tikzpicture}[x=1cm,y=0.35cm]
		\draw[latex-latex] (0,6) -- (0,0) -- (10,0) ;

		\draw[thick] (0,0) -- (-0.5,0);
		\draw[thick] (0,0) -- (0,-0.5);
		
		\draw[thick,red] (1,-2pt) -- (1,2pt);
		\draw[thick,red] (2,-2pt) -- (2,2pt);
		\draw[thick,red] (3,-2pt) -- (3,2pt);
		\draw[thick,red] (4,-2pt) -- (4,2pt);
		\draw[thick,red] (5,-2pt) -- (5,2pt);
		\draw[thick,red] (6,-2pt) -- (6,2pt);
		\draw[thick,red] (7,-2pt) -- (7,2pt);
		\draw[thick,red] (8,-2pt) -- (8,2pt);
		\draw[thick,red] (9,-2pt) -- (9,2pt);
		\draw[thick,red] (-2pt,1) -- (2pt,1);
		\draw[thick,red] (-2pt,2) -- (2pt,2);
		\draw[thick,red] (-2pt,3) -- (2pt,3);
		\draw[thick,red] (-2pt,4) -- (2pt,4);	
		%\draw[thick,red] (-2pt,5) -- (2pt,5);	
		\node at (1,0) [below ,blue]{\footnotesize  $1$};
		\node at (2,0) [below ,blue]{\footnotesize $2$};
		\node at (3,0) [below ,blue]{\footnotesize  $3$};
		\node at (4,0) [below ,blue]{\footnotesize  $4$};
		\node at (5,0) [below ,blue]{\footnotesize  $5$};
		\node at (6,0) [below ,blue]{\footnotesize  $6$};
		\node at (7,0) [below ,blue]{\footnotesize $7$};
		\node at (8,0) [below ,blue]{\footnotesize  $8$};
		\node at (9,0) [below ,blue]{\footnotesize  $9$};
		\node at (0,1) [left ,blue]{\footnotesize  $1$};
		\node at (0,2) [left ,blue]{\footnotesize  $2$};
		\node at (0,3) [left ,blue]{\footnotesize  $3$};
		\node at (0,4) [left ,blue]{\footnotesize  $4$};
		\node at (0,5) [left ,blue]{\footnotesize  $5$};
		\draw[thick, mark = *] plot coordinates{(0,5) (1,2) (3,1) (9,0)};
		\draw[thick, only marks, mark=x] plot coordinates{(1,1) (1,2)   (2,1)     (3,1)     };
		\draw[thick, only marks, mark=*] plot coordinates{  };	
		\node at (0.5,4.2) [above  ,blue]{\footnotesize $S_{11}$};
		\node at (2.5,1.4) [above   ,blue]{\footnotesize $S_{12}$};
		\node at (5.5,0.4) [above   ,blue]{\footnotesize $S_{13}$};
		\end{tikzpicture}
		\caption{  $N_{\ph_1}^{+}(F)$ with respect to $\n_3$.}
	\end{figure}	
\end{example}
\section{Proofs of main results}
After  recalling necessary preliminaries and results in the above section, we are now in the position to  prove our main results. Let us begin by the proof of Theorem \ref{pge3}.
\begin{proof}[Proof of Theorem \ref{pge3}]
Since the degree of $K$ is $7$, by 	 the result of \.{Z}yli\'{n}ski  \cite{Zylinski}, if $p$ divides $i(K)$, then $p<7$,  see also  \cite{Engstrom} by Engstrom. Therefore, the candidate rational  primes to be a common index divisor of $K$ are $2, 3$ and $5$. So, to prove this theorem, it sufficient to show that $3 \nmid i(K)$ and $5 \nmid i(K)$. On the other hand, by  \cite[Proposition 2.13]{Na}, for any $\eta \in \Z_K$, we have the following index formula:
 \begin{equation}\label{indexdiscrininant}
\n_p(D (\eta)) =  2 \n_p(( \Z_K : \Z [\eta])) + \n_p(D_K),
\end{equation}
where $D(\eta)$ is the discriminant of the minimal polynomial of $\eta$ and $D_K$ is the discriminant of $K$. It follows  by Definition $(\ref{Defi(K)})$ of $i(K)$ that  if $p$ divides $i(K)$, then $ p$  divides $ \Delta(F)$. Recall also  that 
\begin{equation}\label{DiscriminantF(x)}
\Delta(F)=-b^4(7^7b^2+2^2\times5^5a^7).
\end{equation}\\
$\bullet$ Let us first show  that  the prime $3$ does not divide $i(K)$.      By Relation  (\ref{indexdiscrininant}) and Formula  $(\ref{DiscriminantF(x)})$, if $3$ divides $i(K)$, then  $$(a, b) \in \{(1, 0), (-1, 0), (1, 1), (-1, 1), (0, 0)\} \md{3}.$$  We distinguish several sub-cases. Table \ref{table2} gives the form of the  factorization of the ideal  $3A_K$ into a product of powers of prime ideals of $A_K$  in all possible  cases. Note also that by $(\ref{Hypothese})$, if $3$ divides both $a$ and $b$, then $\n_3(a)<2 \,\, \text{or} \,\, \n_3(b)<7$. We discus each sub-case separately.
\newpage 
\begin{table}[h!]
	\centering
	\begin{tabular} { | c | c | c |}   
		\hline
		Case & Conditions & Factorization of $3A_K$    \\
		\hline
		A1 & $a \equiv 1 \md{3}, b \equiv 0 \md{3}$ and $5\nmid \n_3(b)$  & $[1^5, 2]$ \\
		\hline
			A2 & $a \equiv 1 \md{3}$ and $5 \mid \n_3(b)$  & $[1, 2, 4]$ \\
		\hline
		A3 & $a \equiv -1 \md{3}, b \equiv 0 \md{3}$ and $5\nmid \n_3(b)$  & $[1, 1, 1^5]$ \\
		\hline
		A4 & $a \equiv -1 \md{3}$ and $5 \mid \n_3(b)$  & $[1, 1, 1, 4]$ \\
		\hline
			A5 & $a \equiv 1 \md{3}, b \equiv 1 \md{3}$ & $[1, 1, 2,3], [1^2, 2, 3]$ or $[2, 2, 3]$ \\
		\hline
			A6 & $a \equiv -1 \md{3}, b \equiv 1 \md{3}$ &$[7]$\\
		\hline
			A7 & $7\n_3(a)>2\n_{3}(b)$ and $\n_3(b) \in \{1, 2, 3, 4, 5, 6\}$ &$[1^7]$\\
		\hline
			A8 & $\n_3(a)=1, \n_3(b)\ge 4 $ and $5 \nmid \n_3(b)$ &$[1^2, 1^5]$ \\
		\hline
			A9 & $\n_3(a)=1$ and $5 \mid \n_3(b)$ &$[1, 1^2, 4]$\\
		\hline
	\end{tabular}
	\caption{The factorization of $3A_K$}
	\label{table2}
\end{table}\textbf{A1:} $a \equiv 1 \md{3}, b \equiv 0 \md{3}$ and $5\nmid \n_3(b)$.  In this case, $F(x) \equiv \ph_1(x)^5 \ph_2(x) \md{3}$, where $\ph_1(x)=x$ and $\ph_2(x)=x^2+1$.   Since $5$ does not divide $\n_3(b)$, $N_{\ph_1}^{+}(F)=S_{11}$ has a single side of degree $1$ joining the points  $(0, \n_3(b))$ and $(5, 0)$ with ramification index $e_{11}=5$. Also, we have  $\n_{\ph_2}(\ol{F(x)})=1$. By Corollary \ref{corollaryore}, we see that $3A_K=\pF_{111}^5\cdot \pF_{211}$, where $\pF_{111}$ and $\pF_{211}$ are two prime ideals of $A_K$ with  respective residue degrees $f(\pF_{111}/3)=1$ and $f(\pF_{211}/3)=2$. In view of Lemma \ref{comindex}, $3$ does not divide $i(K)$.\\
\textbf{A2:} $a \equiv 1 \md{3}$ and $5\mid \n_3(b)$. Here, the factorization of $F(x)$ modulo $3$, the polygons $N_{\ph_1}^{+}(F)$ and  $N_{\ph_2}^{+}(F)$ are the same as in the above case. Further, since $\n_3(b)$ is divisible by $5$, we have 
$$R_{\l_{12}}(F)(y)=
\left
\{\begin{array}{ll} (y+1)(y^4-y^3+y^2-y+1) ,& \mbox{if } b_3 \equiv 1 \md{3},\\
(y-1)(y^4+y^3+y^2+y+1),&\mbox{if } b_3 \equiv -1 \md{3}.
\end{array}
\right.$$Thus, $F(x)$ is $3-$regular. By Theorem \ref{ore}, $3A_K=[1, 2, 4]$. Therefore, by Lemma \ref{comindex}, $3$ does not divide $i(K)$.\\
\textbf{A3:} $a \equiv -1 \md{3}, b \equiv 0 \md{3}$ and $5\nmid \n_3(b)$. In this case, $F(x) \equiv \ph_1(x)^5\ph_2(x)\ph_3(x) \md{3},$ where $\ph_1(x)=x, \ph_2(x)=x-1$ and $\ph_3(x)=x+1$.  Also, $N_{\ph_1}^{+}(F)=S_{11}$ and $R_{\l_{11}}(F)(y)$ are the same as in Case \textbf{A1}. Using Corollary \ref{ore},   $3A_K=[1, 1, 1^5]$. Hence, by Lemma \ref{comindex},  $3$ does not divide $i(K)$.\\
\textbf{A4:} $a \equiv -1 \md{3}$ and $5 \mid \n_3(b)$.  Here, the factorization of $F(x)$ modulo $3$ is  the same as in the above case. Further,  $N_{\ph_1}^{+}(F)=S_{11}$ and $R_{\l_{11}}(F)(y)$ are the same as in Case \textbf{A2}. Therefore, by Corollary \ref{corollaryore},   $3A_K=[1, 1, 1, 4]$. Consequently, by Lemma \ref{comindex},  $3 \nmid i(K)$.\\
\textbf{A5:} $a \equiv 1 \md{3}, b \equiv 1 \md{3}$. In this case, $F(x) \equiv \ph_1(x)^2\ph_2(x)\ph_3(x) \md{3},$ where $\ph_1(x)=x-1, \ph_2(x)=x^2-x-1$ and $\ph_3(x)=x^3-x-1$. By Corollary \ref{corollaryore}(1), the factor $\ph_2(x)$ provides a unique prime ideal $\pF_{211}$ of $A_K$ of residue degree $2$ and of ramification index $1$. Also,  the factor $\ph_3(x)$ provides a unique prime ideal $\pF_{311}$ of $A_K$ of residue degree $3$ and of ramification index $1$.  It follows that $3A_K=\pF_{211}\cdot \pF_{311} \cdot \aF$, where $\aF$ is a proper ideal of $A_K$. By the Fundamental Equality (see $(\ref{FE})$),  the form of the factorization of $\aF$ is either $[1, 1], [1^2]$ or $[2]$. Therefore, the form of the factorization of $3A_K$ is either $[1, 1, 2, 3], [1^2, 2, 3]$ or $[2, 2, 3]$. Hence, by Lemma \ref{comindex}, $3$ does not divide $i(K)$.\\
\textbf{A6:} $a \equiv -1 \md{3}, b \equiv 1 \md{3}$. In this case, $F(x)$ is irreducible modulo $3$. By Corollary \ref{ore}(1), $3A_K$ is a prime ideal. Hence, $3$ does not divide $i(K)$.\\
$\bullet$ Note that by $(\ref{Hypothese})$, if $3$ divides both $a$ and $b$,  then $\n_3(a)<2 \,\, \text{or} \,\, \n_3(b)<7.$ So, the conditions \textbf{A7-A9} cover all possible cases when $a$ and $b$ are both divisible by $3$. On the other hand, $F(x) \equiv \ph_1(x)^7 \md{3}$, where $\ph_1(x)=x$. Thus, $N_{\ph_1}^{+}(F)$ is the lower convex hull of the points $(0, \n_3(b)), (5, \n_3(a))$ and $(7,0)$.\\
\textbf{A7:}  $7\n_3(a)>2\n_{3}(b)$ and $\n_3(b) \in \{1, 2, 3, 4, 5, 6\}$.  In this case, $N_{\ph_1}^{+}(F)=S_{11}$ has a single side of degree $1$ and of ramification index  $7$. By Corollary \ref{corollaryore},  $3A_K=[1^7]$. By Lemma \ref{comindex}, $3$ does not divide $i(K)$.\\
\textbf{A8:} $\n_3(a)=1, \n_3(b)\ge 4 $ and $5 \nmid \n_3(b)$. Here,  $N_{\ph_1}^{+}(F)=S_{11}+S_{12}$ has two sides of degree $1$ joining  the points  $(0, \n_3(b)), (5, 1)$ and $(7,0)$. Their respective ramification indices are: $e_{11}=5$ and $e_{12}=2$. Therefore, by Corollary \ref{corollaryore}(1),  $3A_K=[1^2, 1^5]$. Thus, $3 \nmid i(K)$. \\
 \textbf{A9:} $\n_3(a)=1, \n_3(b)\ge 4 $ and $5 \mid \n_3(b)$. In this case, $N_{\ph_1}^{+}(F)=S_{11}+S_{12}$ has two sides joining  the points  $(0, \n_3(b)), (5, 1)$ and $(7,0)$. Further, we have $d(S_{12})=5, e_{12}=1$ and $R_{\l_{12}}(F)(y)$ is the same as in Case \textbf{A2}. So, $F(x)$ is $3$-regular. Using Theorem \ref{ore}, we get  $3A_K=[1, 1^2, 4]$. So, by Lemma \ref{comindex},  $3$ does not divide $i(K)$.\\
 We conclude that in every case, $3$ is not a common index divisor of $K$.\\
 $\bullet$ Now, we prove that $5$ does not divide $i(K)$. By $(\ref{indexdiscrininant})$ and $(\ref{DiscriminantF(x)})$, if $5$ divides $i(K)$, then $5$ divides $b$. We distinguish seven cases which cover all possible cases. Table \ref{table3} gives the form of the factorization of $5A_K$ in $A_K$.
 \begin{table}[h!]
 	\centering
 	\begin{tabular} { | c | c | c |}   
 		\hline
 		Case & Conditions & Factorization of $5A_K$    \\
 		\hline
 		B1 & $7\n_5(a)>2\n_{5}(b)$ and $\n_5(b) \in \{1, 2, 3, 4, 5, 6\}$ &$[1^7]$ \\
 	\hline
 	B2 & $\n_5(a)=1, \n_5(b)\ge 4 $ and $5 \nmid \n_5(b)$ &$[1^2, 1^5]$ \\
 	\hline
 	B3 & $\n_5(a)=1$ and $5 \mid \n_5(b)$ &$[1^7]$ \\
 	\hline
 		B4 & $a \equiv \pm 2 \md{5}, b \equiv 0 \md{5}$, and $5 \nmid \n_5(b)$ &$[1^5, 2]$ \\
 	\hline
 		B5 & $a \equiv \pm 2 \md{5}$ and $5$ divides $ \n_5(b)$ &$[2,  1^5]$ or $[1, 1^4, 2]$\\
 	\hline
 	%	B6 & $a \equiv \pm 2 \md{5}, \n_5(b)=5^kc$, and $ac^2+b_5 \equiv 0 \md{25}$ &$[1, 1^4,  5]$ \\
 	%\hline
 		B6  & $a \equiv \pm 1 \md{5}, b \equiv 0 \md{5}$, and $5 \nmid \n_5(b)$ &$[1, 1, 1^5]$ \\
 	\hline
 		B7 & $a \equiv \pm 1 \md{5}$ and $5$ divides $ \n_5(b)$ &$[1, 1, 1^5]$ or $[1, 1, 1, 1^4]$ \\
 	\hline
 	\end{tabular}
 	\caption{The factorization of $5A_K$}
 	\label{table3}
 \end{table}\\
\textbf{B1, B2:} These cases are respectively similar to \textbf{A7, A8} of Table \ref{table2}.\\
\textbf{B3:} $\n_5(a)=1$ and $5 \mid \n_5(b)$. In this case $F(x) \equiv \ph_1(x)^7 \md{5}$.  In order to apply Ore's Theorem (Theorem \ref{ore}), we replace the lifting $\ph_1(x)=x$ of $\ol{\ph_1(x)} = \ol{x} \in \F_5[x]$ by  $\psi_1(x)=x-5^kc$ for some positive integers $c$ and $k$,  which allows to  the polynomial $F(x)$ to be $\psi_1$-regular. For any rational prime $p$, it is important to note that Theorem \ref{ore} does not depend on the monic irreducible  liftings of the monic irreducible factors of $F(x)$ modulo $p$.  The $\psi_1$-adic development of $F(x)$ is 
\begin{eqnarray}\label{Devforp=5}
F(x)&=& 5^{7k}c^7+5^{5k}ac^5+b+5^{4k+1}c^4(7 \cdot 5^{2k-1} c^2 + a) \psi_1(x)+ 5^{3k+1}c^3(21 \cdot 5^{k-1}c+2a) \psi_1(x)^2 \nonumber  \\ 
&+& 5^{2k+1}c^2 (7 \cdot 5^{2k}c^2+2a) \psi_1(x)^3+ 5^{k+1} (7 \cdot 5^{2k}+a) \psi_1(x)^4+ (21+5^{5k}c^2+a) \psi_1(x)^5 \nonumber \\
&+&7 \cdot 5^k c \psi_1(x)^6+ \psi_1(x)^7.
\end{eqnarray}
Let $a_0= 5^{7k}c^7+5^{5k}ac^5+b, a_1=5^{4k+1}c^4(7 \cdot 5^{2k-1} c^2 + a), a_2=5^{3k+1}c^3(21 \cdot 5^{k-1}c+2a), a_3= 5^{2k+1}c^2 (7 \cdot 5^{2k}c^2+2a), a_4=5^{k+1} (7 \cdot 5^{2k}+a), a_5=21 \cdot 5^{5k}c^2+a, a_6=7 \cdot 5^k c$, and $\mu_i=\n_5(a_i)$ for $i=1, \ldots, 6$.\\
If $\n_5(b)=5$, by choosing  $k\ge 2$, then we get $\mu_0= \n_5(b)=5, \mu_1 \ge 9, \mu_2 \ge 7, \mu_3 \ge 5, \mu_4 \ge 3, \mu_5=1$ and $\mu_6 \ge 2$. Thus, by $(\ref{Devforp=5})$,   $N_{\psi_1}^{+}(F)=S_{11}$ has a single side of degree $1$ joining the points $(0, 5)$ and $(7, 0)$. Its ramification index equals $7$. By Corollary \ref{corollaryore}(2),   $3A_K=[1^7]$. When $\n_5(b) \ge 10$, we choose $k=1$, then we obtain the same polygon, and the same factorization of $3A_K$ as  when  $\n_{5}(b)=5$. 	It follows that $5 \nmid i(K)$.\\
\textbf{B4:} $a \equiv \pm 2 \md{5}, b \equiv 0 \md{5}$ and $5 \nmid \n_5(b)$. Here, $F(x) \equiv \ph_1(x)^5 \ph_2(x) \md{5},$ where $\ph_1(x)=x$ and $\ph_2(x)=x^2+a$. By Corollary \ref{corollaryore}(1), the factor $\ph_2(x)$ provides a unique prime ideal of $A_K$ lying above $5$ of residue degree $2$ with ramification index $1$, say $\pF_{211}$.  Since, $5 \nmid \n_5(b)$, by Corollary \ref{corollaryore}(2),  $\ph_1(x)$ provide a unique prime  ideal of $A_K$ lying above $5$ of residue degree $1$ with ramification index $5$, say $\pF_{111}$. Therefore, $5A_K=[1^5, 2]$.\\
\textbf{B5:}  $a \equiv \pm 2 \md{5}$ and $5$ divides $ \n_5(b)$. Set $\n_5(b)=5k$. Let $C_{a,b_5} \in \Z$ such that $5$ divides  $aC_{a,b_5}^5+b_5$. To treat this case,  we use $\psi_1(x)=x-5^kC_{a,b_5}$ as in Case \textbf{B3}.   Write  $a_0=5^{7k}C^7_{a,b_5}+5^{5k}(aC_{a,b_5}^5+b_5)$.   According to $(\ref{Devforp=5})$,  we have  $\mu_1=4k+1, \mu_2=3k+1, \mu_3=2k+1$ and  $\mu_4=k+1$. We distinguish two cases.  If   $\n_5(aC_{a,b_5}^5+b_5)=1$, then  $\mu_0=5k+1$. Thus, by $(\ref{Devforp=5})$, $N_{\psi_1}^{+}(F)=S_{11}$ has a single side of degree $1$ joining the points $(0, 5k+1)$ and $(5, 0)$. Its ramification index equals $5$. By Corollary \ref{corollaryore},  the factor $\psi_1(x)$ (or $\ph_1(x)$) provides   a unique prime ideal of $A_K$ lying above $5$ of residue degree $1$ with ramification index $1$. Therefore,   $5A_K=[2, 1^5]$. 
If   $\n_5(aC_{a,b_5}^5+b_5) \ge 2$, then  $\mu_0 \ge 5k+2$. Thus, by $(\ref{Devforp=5})$, $N_{\psi_1}^{+}(F)=S_{11}+S_{11}$ has two sides of degree $1$ each joining the points $(0, \mu_0), (1, 4k+1)$ and $(5, 0)$. Their ramification indices are $e_{11}=1, e_{12}=4$. By Theorem \ref{ore},  the factor $\psi_1(x)$ (or $\ph_1(x)$) provides   two prime ideals of $A_K$ lying above $5$ of residue degree $1$ each. Therefore,  $5A_K=[1, 1^4, 2]$. \\
\textbf{B6, B7:} We proceed analogously and respectively as in Cases \textbf{B4, B5}. \\
Since the factorization of $5A_K$ does not satisfies the inequality  $L_5(f) > N_5(f)$ for any positive integer $f$. We conclude by Lemma \ref{comindex} that $5$ does not divide $i(K)$.
This completes the proof of the theorem.
\end{proof}
From Theorem \ref{pge3}, if $p \ge 3$, then $p
 \nmid i(K)$. 	Therefore, $i(K)=2^{\n_2(i(K))}$.
Now, let us prove Theorem \ref{p=2}. In every case, we give the form of  the factorization of $2A_K$ and the value of $i(K)$.
\begin{proof}[Proof of Theorem \ref{p=2}]\
%In every case we give the factorization of the ideal $2A_K$.\\
\\ \textbf{C1:}  $a \equiv 1 \md{2}$ and $b \equiv 1 \md{2}$. In this case, $F(x) \equiv  (x^2 + x + 1) (x^5 + x^4 + x^3 + x + 1) \md{2}$.	 By Corollary \ref{corollaryore}(1), $2A_K=[2, 5]$. In view of Lemma \ref{comindex}, $2$ does not divide $i(K)$. So, $i(K)=1$.\\
\textbf{C2:} $a \equiv 0 \md{2}$ and $b \equiv 1 \md{2}$. Here, $F(x) \equiv (x + 1) (x^3 + x + 1) (x^3 + x^2 + 1) \md{2}$. By Corollary \ref{corollaryore}(1),  $2A_K=[1, 3, 3]$. Therefore, $\n_2(i(K))=0$. So, $i(K)=1$.\\
$\bullet$ In Cases \textbf{C3, C4, C5}, $2$ divides both $a$ and $b$. These cases  are respectively similar to Cases $A7, A8, A9$ when we consider $p=3$.  We omit their proofs. In these cases $i(K)=1$.\\
$\bullet$ From Case \textbf{C6},  $2$ divides $b$, but does not divide $a$. It follows that $F(x) \equiv \ph_1(x)^5\ph_2(x)^2 \md{2}$, where $\ph_1(x)=x$ and $\ph_2(x)=x-1$.  \textbf{For $\ph_1(x)$},  the polynomial $F(x)$ is $\ph_1$-regular. Moreover, by using Theorem \ref{ore},  we have the following:
\begin{enumerate}
\item If  $5$ does not divide $\n_2(b)$, then $\ph_1(x)$ provides a unique prime ideal lying above $2A_K$, of residue degree $1$ with  ramification index $5$. So, $2A_K=\pF_{111}^5\cdot\aF$, where $f(\pF_{111}/2)=1$ and $\aF$ is a non-zero ideal of $A_K$.
\item 	If  $5$ does divides $\n_2(b)$, then $\ph_1(x)$ provides two  prime ideals lying above $2A_K$ with ramification index $1$ each. One of them has a residue degree $1$, and the other has a residue degree  $4$.  Thus, $2A_K=\pF_{111}\cdot \pF_{121} \cdot \aF$, where $f(\pF_{111}/2)=1, f(\pF_{112}/2)=4$ and  $\aF$ is a non-zero ideal of $A_K$.
\end{enumerate}
Thus,  the number of prime ideals of $A_K$ that  divide $2A_K$, and  provided from $\ph_1(x)$, are determined with their residue degrees. On the other hand, \textbf{the ideal $\aF$ is provided from the factor $\ph_2(x)$. To factorize it, we analyze $N_{\ph_2}^{+}(F)$, the $\ph_2$-principal  Newton polygon of $F(x)$.} The $\ph_2$-adic development of $F(x)$ is  
\begin{equation}\label{Devph2a1b0}
F(x)=1+a+b+(7+5a)\ph_2(x)+(21+10a)\ph_2(x)^2+\ldots+\ph_2(x)^7
\end{equation}
Let $\n = \n_{2}(1+a+b)$ and $\mu= \n_2(7+5a)$. It follows by  $(\ref{Devph2a1b0})$  that  $N_{\ph_2}^{+}(F)$ is the lower convex hull of the points $(0, \n), (1, \mu)$ and $(2,0)$. Note also that the finite residual field $\F_{\ph_2}$ is isomorphic to $\F_2$. \\
\textbf{C6, C7, C8:}  In all these cases, we have $\n = 2$ and $\mu=1$. Thus,  $N_{\ph_2}^{+}(F)=S_{11}$ has a single side of degree $2$ joining the points $(0, 2), (1, 1)$ and $2, 0$. Further, we have $e_{11}$ and $R_{\l_{11}}(F)(y)=1+y+y^2$ which is  separable  in  $F_{\ph_2}[y]$. Therefore, by Theorem \ref{ore},    the form of the factorization of $\aF$ is $[2]$. Thus, we conclude the form of the factorization of $2A_K$ in these cases as given in Table \ref{table1}.\\
\textbf{C9, C10, C11:}  In all these cases, $\n \ge 2$ and $\mu=1$. Thus,  $N_{\ph_2}^{+}(F)=S_{11}+S_{12}$ has two sides of degree $1$ each, joining
 the points $(0, \n), (1, 1)$ and $(2, 0)$. Their ramification indices equal $1$. Their attached residual polynomial are separable as they are of degree $1$. By Theorem \ref{ore},      the form of the factorization of $\aF$ is $[1, 1]$. Therefore, we conclude the form of the factorization of $2A_K$ in these cases as given  in Table \ref{table1}. Since $L_2(1)=3>2 =N_{2}(1)$, by Lemma \ref{comindex}, $2 \mid i(K)$. In Case \textbf{C11}, we have  $2A_K=[1, 1, 1, 4]$, then according to  Engstrom's table concerning the index of number fields of degrees less than  or equal $7$ (see \cite[Page 234]{Engstrom}), we see that $\n_2(i(K))= 1$. On the other hand, in Cases \textbf{C9} and \textbf{C10}, we have  $2A_K=[1, 1, 1^5]$. Unfortunately,   Engstrom's table does not give $\n_2(i(K))$ for this factorization form. Using Theorem \ref{ore}(1),   one has: $$\n_2((A_K: \Z[\th]))= ind_{\ph_1}(F)+ind_{\ph_2}(F)= ind_{\ph_1}(F)+ind_{\psi_2}(F)=0+1=1.$$
 On the other hand, according to Definition \ref{Defi(K)}, we have $\n_{2}(i(K)) \le \n_2((A_K: \Z[\th]))$. Thus, $\n_{2}(i(K))=1$, and so $i(K)=2$ in \textbf{C9} and  \textbf{C10}.  \\
 \textbf{C12, C13, C14:} In all these cases, $\n=1$. Thus, $N_{\ph_2}^{+}(F)=S_{11}$ has a single side of degree $1$ with ramification index $2$. By Corollary  \ref{corollaryore}(1),  the form of the factorization of $\aF$ is $[1^2]$. Therefore,  $2A_K=[1^2, 1^5]$ or $2A_K=[1, 1^2, 4]$. Hence, by Lemma \ref{comindex}, $2 \nmid i(K)$. \\
 $\bullet$ From Case \textbf{C15}, we have $\n_2(1+a)=\n_2(b)=1$. It follows that $\min \{\n, \mu\} \ge 2$. Then we cannot control their values. To apply Ore's Theorem (Theorem \ref{ore}), we replace the lifting $\ph_2(x)=x-1$ of $\ol{\ph_2(x)}$ by  $\psi_2(x)=x-s$ for an adequate odd rational integer $s$ which allows to  the polynomial $F(x)$ to be $\psi_2$-regular. The $\psi_2$-adic development of $F(x)$ is  \begin{equation}\label{Devph2a1b0psi2}
 F(x)=s^7+as^5+b+(7s^6+5as^4)\psi_2(x)+(21s^5+10as^3)\psi_2(x)^2+\ldots+\psi_2(x)^7.
 \end{equation}
 Let $\omega=\n_{2}(s^7+as^5+b)$ and $\delta = \n_2(7s^6+5as^4)$. Thus, by $(\ref{Devph2a1b0psi2})$, $N_{\psi_2}^{+}(F)$ is the lower convex hull of the points $(0, \omega), (1, \delta)$ and $(2,0)$. In the next cases, we give  $s$ explicitly, and  the form of the factorization of $\aF$ in $A_K$. Remark in these cases that the factor $\ph_1(x)$ provide a unique prime ideal of residue degree $1$ with ramification index $5$, because $\n_{2}(b)=1$.\\
 \textbf{C15:} $(a, b) \in \{(1, 10), (9, 2), (1, 6), (9, 14)\} \md{16}$. When $(a, b) \in \{(1, 10), (9, 2)\} \md{16}$, we choose any  $s$ such that $s \equiv 3, 7, 11, \, \mbox{or} \, 15  \md{16}$, and if  $(a, b) \in \{ (1, 6), (9, 14)\} \md{16}$, consider any $s$ satisfying $s \equiv 1, 5, 9, \, \mbox{or} \, 13 \md{16}$.  Then, we get $\omega = 3$ and $\delta = 2$. It follow by $(\ref{Devph2a1b0psi2})$ that $N_{\psi_2}^{+}(F)=S_{11}$ has a single side of degree $1$ joining the points $(0, 3)$ and $(2, 0)$. Its ramification index equals $2$. By Corollary \ref{corollaryore},  the form of the factorization of $\aF$ is $[1^2]$. Therefore, $2A_K=[1^2, 1^5]$. So, $2 \nmid i(K)$.\\
	  \textbf{C16:}   $(a, b) \in \{(1, 18), (17, 2), (1, 14), (17, 30)\} \md{32}$. For  $(a, b) \in \{(1, 18), (17, 2)\} \md{16}$, we choose $s \equiv 3, 7, 11, 15, 19, 23, 27, 31 \md{32}$, and for  $(a, b) \in \{ (1, 14), (17, 30)\} \md{32}$, we choose any  $s\equiv 1, 5, 9, 13, 17, 21, 25, 29 \md{32}$. Then, we have $\omega=4$ and $\delta = 2$. It follows by $(\ref{Devph2a1b0psi2})$ that $N_{\psi_2}^{+}(F)=S_{11}$ has a single side of degree $2$ joining the points $(0, 4), (1, 2)$ and $(2, 0)$. Its attached residual polynomial is $R_{\l_{11}}(F)(y)=y^2+y+1$ which is separable in $\F_{\psi_2}[y]$. So, $F(x)$ is $\psi_2$-regular. By Theorem \ref{ore},  the form of the factorization of $\aF$ is $[2]$. Therefore, $2A_K=[2, 1^5]$. Hence, by Lemma \ref{comindex}, $2 \nmid i(K)$. \\
 \textbf{C17:} $(a, b) \in \{(1, 2), (17, 18), (1, 30), (17, 14)\}\md{32}$. When $(a, b) \in \{(1, 2), (17, 18)\}\md{32}$, let $s \equiv 3, 7, 11, 15, 19, 23, 27, 31 \md{32}$, and for$(a, b) \in \{ (1, 30), (17, 14)\}\md{32}$, let $s \equiv  1, 5, 9, 13, 17, 21, 25, 29  \md{32}$. Under these considerations, we have $\omega \ge 5$ and $\delta=2$. Thus, $N_{\psi_2}^{+}(F)=S_{11}+S_{12}$ has two sides of degree $1$ each joining the points $(0, \omega), (1, 2)$ and $(2, 0)$. Their ramification indices equal $1$. Using Theorem \ref{ore},  the form of the factorization of $\aF$ is $[1, 1]$. So, $2A_K=[1, 1, 1^5]$. Therefore, $2 \mid i(K)$.  By Theorem \ref{ore}(1),   one gets: $$\n_2((A_K: \Z[\th]))=  ind_{\ph_1}(F)+ind_{\psi_2}(F)=0+2=2.$$ So,  $\n_{2}(i(K)) \le 2$. Hence $i(K)=2 \, \mbox{or} \, 4$.\\ 
% \textbf{C18:} $a \equiv 17 \md{32}$ and $b \equiv 14, 18 \md{32}$. For  $b \equiv 14 \md{32}$, let $s \equiv 1 \md{32}$, and for  $b \equiv 18 \equiv \md{32}$, let $s \equiv 3 \md{32}$. Then, we are in the same situation as in Case \textbf{C16}.\\
% \textbf{C19:}  $a \equiv 17 \md{32}$ and $b \equiv 2, 30 \md{32}$. For  $b \equiv 14 \md{32}$, let $s \equiv 1 \md{32}$, and for  $b \equiv 30 \md{32}$, let $s \equiv 3 \md{32}$. Then, we are in the same situation as in Case \textbf{C17}.\\
% \textbf{C20:} $a \equiv 9 \md{16}$ and $b \equiv 2, 14 \md{16}$. When $b \equiv 2 \md{16}$, consider $s \equiv 3 \md{16}$, and for $b \equiv 14 \md{16}$, consider $s \equiv 13 \md{16}$. Then, we obtain that $\omega=3$ and $\delta=2$. So this case is similar to Case \textbf{C15}.
\textbf{C18:}  $(a, b) \in \{(5, 2), (13, 10), (5, 14), (13, 6)\}\md{16}$. For $(a, b) \in \{(5, 2), (13, 10)\}\md{16}$, we choose any $s \equiv 1, 5, 9, 13 \md{16}$, and for $(a, b) \in \{(5, 14), (13, 6)\}\md{16}$,  we consider any $s \equiv 3, 7, 11, 15 \md{16}$. Then, we get $\omega = 3$ and $\delta \ge 3$. It follows  that $N_{\psi_2}^{+}(F)=S_{11}$ has a single side of degree $1$ joining the points $(0, 3)$ and $(2, 0)$. As in Case \textbf{C15}, $2A_K= [1^2, 1^5]$. Consequently, $2 \nmid i(K)$.\\
This completes the proof of the theorem, and so $i(K)=1$.
\end{proof}
%\begin{thebibliography}{99}	

\end{document}